\documentclass[english]{article}
\usepackage[T1]{fontenc}
\usepackage{float}
\usepackage{booktabs}
\usepackage{units}
\usepackage{amsmath}
\usepackage{amsthm}
\usepackage{amssymb}
\usepackage{graphicx}
\usepackage{babel}
\providecommand{\keywords}[1]{\textbf{\textit{Keywords : }} #1}
\providecommand{\smcn}[1]{\textbf{\textit{Mathetical classification numbers : }} #1}

\makeatletter

\providecommand{\tabularnewline}{\\}

\numberwithin{equation}{section}
\newcommand{\lyxaddress}[1]{
	\par {\raggedright #1
	\vspace{1.4em}
	\noindent\par}
}
\theoremstyle{remark}
\newtheorem*{notation*}{\protect\notationname}
\theoremstyle{plain}
\newtheorem{thm}{\protect\theoremname}[section]
\theoremstyle{definition}
\newtheorem{defn}[thm]{\protect\definitionname}
\theoremstyle{remark}
\newtheorem{rem}[thm]{\protect\remarkname}
\theoremstyle{plain}
\newtheorem{cor}[thm]{\protect\corollaryname}
\theoremstyle{definition}
\newtheorem{example}[thm]{\protect\examplename}
\theoremstyle{plain}
\newtheorem{lem}[thm]{\protect\lemmaname}
\theoremstyle{plain}
\newtheorem{question}[thm]{\protect\questionname}

\makeatother

\usepackage{babel}
\providecommand{\corollaryname}{Corollary}
\providecommand{\definitionname}{Definition}
\providecommand{\examplename}{Example}
\providecommand{\lemmaname}{Lemma}
\providecommand{\notationname}{Notation}
\providecommand{\questionname}{Question}
\providecommand{\remarkname}{Remark}
\providecommand{\theoremname}{Theorem}

\begin{document}
\title{The graphs of pyramids are determined by their spectrum}
\author{Noam Krupnik \thanks{Noam.Krupnik@campus.technion.ac.il} \and Abraham
Berman \thanks{berman@technion.ac.il}}
\maketitle

\lyxaddress{Department of Mathematics, Technion -- Israel Institute of Technology,
Haifa 3200003, Israel}
\begin{abstract}
For natural numbers $k<n$ we study the graphs $T_{n,k}:=K_{k}\lor\overline{K_{n-k}}$.
For $k=1$, $T_{n,1}$ is the star $S_{n-1}$. For $k>1$ we refer
to $T_{n,k}$ as a \emph{graph of pyramids}. We prove that the graphs
of pyramids are determined by their spectrum, and that a star $S_{n}$
is determined by its spectrum iff $n$ is prime. We also show that
the graphs $T_{n,k}$ are completely positive iff $k\le2$. 
\end{abstract}

\keywords { 
cospectral graphs, graphs that are determined by their spectrum,
graphs of pyramids , star graphs, completely positive graphs, Schur
complement. 
} \\

\smcn{05C50, 15B48}

\section{Introduction }

The first author participated in a course in matrix theory given by
the second author. This paper is one of the outcomes of this course.
Two of the topics studied in the course were completely positive (CP)
matrices and graphs and cospectrality of graphs. In one of the problems
in the course we used the Schur complement to show that the graph
$T_{n}$, the graph of $n-2$ triangles with a common base, is CP.
In this paper, we generalize $T_{n}$ to $T_{n,k}$ - graphs that
consist of $n-k$ pyramids $K_{k+1}$ with $K_{k}$ as a common base
and use the Schur complement to compute their spectrum. For $k>1$
we refer to $T_{n,k}$ as graphs of pyramids and show that they are
determined by their spectrum (DS). For $k=1$, $T_{n,1}$ is a star
$S_{n-1}$, and it is DS if and only if $n-1$ is prime. Since the
motivation for this article was the proof that $T_{n}$ is completely
positive, we start the paper with a short survey of complete positivity.
We summarize the paper with a table that classifies the graphs $T_{n,k}$
according to their being DS/CP. 

\begin{figure}[H]
\begin{centering}
\includegraphics[scale=0.3]{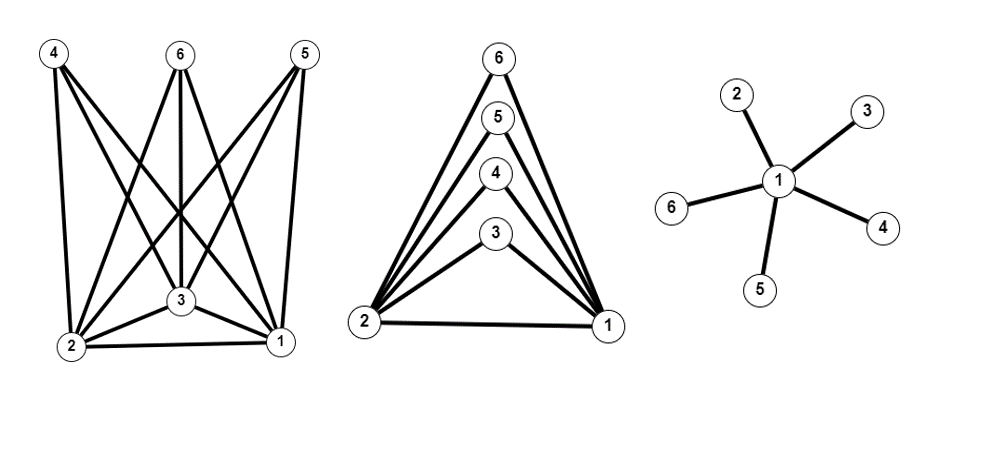}
\par\end{centering}
\centering{}\caption{the graphs $T_{6,3}$ , $T_{6\text{,}2}$ and $T_{6,1}$}
\end{figure}

\section{Background }

\subsection{Matrix theory preliminaries}

We denote by $\mathbb{R}^{m\times n}$ the vector space of $m\times n$
real matrices, and by $S^{n}$ the subspace of symmetric matrices
in $\mathbb{R}^{n\times n}$. Let $\lambda_{k}\le...\le\lambda_{2}\le\lambda_{1}$
be the distinct eigenvalues of $A\in S^{n}$ and let $\alpha_{i}$
be the multiplicity of $\lambda_{i}\ ;\ i=1,2,...,k$ , $\sum_{i=1}^{k}\alpha_{i}=n$.
The \emph{spectrum} of $A$ , $\sigma\left(A\right)$, is the multiset
of its eigenvalues, and is denoted by $\sigma\left(A\right)=\left\{ \left(\lambda_{k}\right)^{\alpha_{k}},\ldots,\left(\lambda_{2}\right)^{\alpha_{2}},\left(\lambda_{1}\right)^{\alpha_{1}}\right\} $. 

For background on matrix theory we refer the readers to \cite{horn2012matrix}.
Two theorems that are frequently used in the paper are Cauchy's interlace
theorem (Theorem \ref{thm:Cauchy's-Eigenvalue-interlacing}) and Schur's
theorem (Theorem \ref{thm:Schur complement}). 
\begin{notation*}
$\left\langle n\right\rangle =\left\{ 1,2,...,n\right\} $
\end{notation*}
\begin{thm}
\label{thm:Cauchy's-Eigenvalue-interlacing} (Cauchy's Interlace Theorem
\cite{parlett1998symmetric}). For natural numbers $n>m$, let $B\in S^{m}$
be a principal submatrix of $A\in S^{n}$. Denote the eigenvalues
of $A$ by $\lambda_{n}\le...\le\lambda_{2}\le\lambda_{1}$ and the
eigenvalues of $B$ by $\mu_{m}\le...\le\mu_{2}\le\mu_{1}$. Then
\[
\forall k\in<m>\ ,\ \lambda_{n+k-m}\le\mu_{k}\le\lambda_{k}.
\]
\end{thm}

\begin{defn}
\label{def:Schur-Complement--}Let $M=\begin{pmatrix}A & B\\
C & D
\end{pmatrix}$ be a square block matrix, where $D$ is nonsingular. The \emph{Schur Complement},
$\nicefrac{M}{D}$, of $D$ \emph{in} $M$ is 
\[
\nicefrac{M}{D}=A-BD^{-1}C.
\]
\end{defn}

\begin{thm}
\label{thm:Schur complement} \cite{schur1917potenzreihen} In the
notation of Definition \ref{def:Schur-Complement--} 
\begin{enumerate}
\item $det\left(M\right)=det\left(D\right)det\left(\nicefrac{M}{D}\right)$.
\item $rank\left(M\right)=rank\left(D\right)+rank\left(\nicefrac{M}{D}\right)$.
\end{enumerate}
\end{thm}

\subsection{Graph theory preliminaries}

A \emph{graph} $G$ is a pair $G=\left(V,E\right)$ where $V$ is
the set of \emph{vertices} and $E\subset V\times V$ is the set of
\emph{edges}. The \emph{order} of a graph $G=\left(V,E\right)$
is $\left|V\right|$. A graph $G$ is \emph{simple} if it has no
loops, i.e. for all $v\in V$ , $\left(v,v\right)\notin E$ . $G$
is an \emph{undirected} graph if for every $v,u\in V$, $\left(v,u\right)\in E\ \Leftrightarrow\left(u,v\right)\in E$.
In this paper all of the graphs are simple and undirected. A graph
is \emph{connected} if there is a path between any two of its vertices.
A \emph{tree} is a connected graph without cycles. 
\begin{thm}
\label{thm:Tree-properties}The following properties of a graph $G=\left(V,E\right)$
of order $n$ are equivalent
\begin{enumerate}
\item $G$ is a tree.
\item $G$ is connected and $\left|E\right|=n-1$.
\item $G$ has no cycles and $\left|E\right|=n-1$.
\end{enumerate}
\end{thm}

\begin{defn}
The \emph{complement}, $\overline{G}$, of a graph $G=\left(V,E\right)$,
is: 
\[
\overline{G}=\left(V,\left\{ \left(v,u\right)\ \mid\ v\ne u,\ \left(v,u\right)\notin E\right\} \right).
\]
\end{defn}

\begin{defn}
Let $G_{1}=\left(V_{1},E_{1}\right),G_{2}=\left(V_{2},E_{2}\right)$
be two graphs. 
\begin{enumerate}
\item $G_{1}$ and $G_{2}$ are \emph{disjoint} if $V_{1}\cap V_{2}=\emptyset$. 
\item If $G_{1},G_{2}$ are disjoint, their \emph{disjoint union} is 
\[
G_{1}\uplus G_{2}=\left(V_{1}\uplus V_{2},E_{1}\uplus E_{2}\right).
\]
\end{enumerate}
\end{defn}

\begin{defn}
The \emph{join} , $G_{1}\lor G_{2}$, of two graphs $G_{1}=\left(V_{1},E_{1}\right),G_{2}=\left(V_{2},E_{2}\right)$
is obtained from their disjoint union by connecting all the vertices
in $G_{1}$ to all the vertices in $G_{2}$. 
\end{defn}

In the paper we consider the following graphs of order $n$ : 
\begin{itemize}
\item $K_{n}$ - The complete graph. 
\item $\overline{K_{n}}$ - The empty graph. 
\item $S_{n-1}$ - The star graph. 
\item $C_{n}$ - The cycle. 
\item $P_{n}$ - The path. 
\end{itemize}
We also consider the complete bipartite graph $K_{n,m}=\overline{K_{n}}\lor\overline{K_{m}}$.
Observe that $S_{n}=K_{1,n}$. 

\begin{defn}
Let $G=\left(V,E\right)$. The \emph{line graph} of $G$ is $L\left(G\right)=\left(V',E'\right)$
where 
\[
V'=E\ ,\ E'=\left\{ \left(\left(v_{1},v_{2}\right),\left(v_{2},v_{3}\right)\right)\ \mid\ v_{1},v_{2},v_{3}\in V\right\} .
\]
\end{defn}

\begin{defn}
Let $G=\left(V\left(G\right),E\left(G\right)\right)$ and $H=\left(V\left(H\right),E\left(H\right)\right)$. 
\begin{enumerate}
\item $H$ is a \emph{subgraph} of $G$ if $V\left(H\right)\subseteq V\left(G\right)$
and $E\left(H\right)\subseteq E\left(G\right)$. 
\item $H$ is an \emph{induced subgraph} of $G$ if $H$ is a subgraph
of $G$ and $E\left(H\right)$ contains every edge in $G$ that has
both ends in $V\left(H\right)$ . 
\end{enumerate}
\end{defn}

\begin{defn}
Let $G=\left(V,E\right)$ be a graph and let $U\subseteq V$ be a
set of vertices.
\begin{enumerate}
\item $U$ is a \emph{clique} if the subgraph induced by $U$ is a complete
graph. 
\item $U$ is a \emph{maximal clique} if it is a clique, and $G$ does
not have a clique of greater order. 
\end{enumerate}
\end{defn}

\subsection{Spectral graph theory preliminaries}

There are several books on spectral graph theory, for example \cite{brouwer2011spectra}
, \cite{cvetkovic1988recent} , \cite{cvetkovic1980spectra}. 
\begin{defn}
The adjacency matrix of a graph $G=\left(V,E\right)$ with $n$ vertices,
is a $\left(0,1\right)$ matrix $A\left(G\right)=\left(a_{i,j}\right)\in S^{n}$,
where $a_{i,j}=1$ if and only if $\left(i,j\right)\in E$. 
\end{defn}

\begin{rem}
There are other matrices associated with a graph $G$. For example
the Laplacian, the signless Laplacian, the normalized Laplacian and
the distance matrix. We do not consider them in this paper. 
\end{rem}

\begin{defn}
The \emph{spectrum of a graph} $G$ , $\sigma\left(G\right)$, is
the spectrum of $A\left(G\right)$. The \emph{charateristic matrix (polymonial) of}
$G$ is the characteristic matrix (polynomial) of $A\left(G\right)$. 
\end{defn}

\begin{thm}
\label{thm:spectral-properties-of-a-graph}Let $G=\left(V,E\right)$
be a graph with eigenvalues $\lambda_{n}\le...\le\lambda_{2}\le\lambda_{1}$.
The number of close walks of length $k$ in $G$ is $trace(A^{k})=\sum_{i=1}^{n}\lambda_{i}^{k}$. 
\end{thm}

\begin{cor}
\label{cor:num_of_adges_and_triangels}$ $
\begin{enumerate}
\item $\sum_{i=1}^{n}\lambda_{i}=0$. 
\item The number of edges in $G$ is $\frac{1}{2}\sum_{i=1}^{n}\lambda_{i}^{2}$.
\item The number of triangles in $G$ is $\frac{1}{6}\sum_{i=1}^{n}\lambda_{i}^{3}$.
\end{enumerate}
\end{cor}

\begin{thm}
\label{claim:Spectrum-of-disjoint} The spectrum of disjoint union
of two graphs $H_{1},H_{2}$ is $\sigma_{H_{1}\uplus H_{2}}=\sigma_{H_{1}}+\sigma_{H_{2}}$,
where the $+$ sign stands for multisets sum. 
\end{thm}

\begin{notation*}
For a real number $a$ and a graph $G$, let $\nu\left(G,a\right)$
($\mu\left(G,a\right)$) denote number of eigenvalues of $G$ that
are greater (smaller) than or equal to $a$. 
\end{notation*}
The following important theorem is a corollary of Theorem \ref{thm:Cauchy's-Eigenvalue-interlacing}: 
\begin{thm}
\label{thm:induced-subgraph-interlacing} Let be $H$ is an induced
subgraph of $G$. Let $\lambda_{n}\le...\le\lambda_{2}\le\lambda_{1}$
be the eigenvalues of $G$ and $\mu_{m}\le...\le\mu_{2}\le\mu_{1}$
the eigenvalues of $H$. Then
\begin{enumerate}
\item $\mu_{1}\le\lambda_{1}$ and $\lambda_{n}\le\mu_{m}$.
\item For every real number $a$, $\nu\left(G,a\right)\ge\nu\left(H,a\right)$
and $\mu\left(G,a\right)\ge\mu\left(H,a\right)$.
\item $rank\left(A\left(H\right)\right)\le rank\left(A\left(G\right)\right)$.
\end{enumerate}
\end{thm}

\begin{thm}
\label{thm:spectrum-of-well} The following table lists the spectrum
of some special graphs.
\end{thm}

\begin{center}
\begin{tabular}{ccc}
\toprule 
\textbf{Name} & \textbf{Symbol} & \textbf{Spectrum}\tabularnewline
\midrule 
Complete graph & $K_{n}$ & $\sigma\left(K_{n}\right)=\left\{ \left(-1\right)^{n-1},\left(n-1\right)\right\} $\tabularnewline
Complete bipartite graph & $K_{n,m}$ & $\sigma\left(K_{m}\right)=\left\{ \left(-\sqrt{n\cdot m}\right),\left(0\right)^{n+m-2},\left(\sqrt{n\cdot m}\right)\right\} $\tabularnewline
Star & $S_{n}$ & $\sigma\left(S_{n}\right)=\left\{ \left(-\sqrt{n}\right),\left(0\right)^{n-1},\left(\sqrt{n}\right)\right\} $\tabularnewline
Path & $P_{n}$ & $\sigma\left(P_{n}\right)=\left\{ 2\cdot cos\left(\frac{\pi\cdot k}{n+1}\right)\ \mid\ k\in\left\langle n\right\rangle \right\} $\tabularnewline
Cycle & $C_{n}$ & $\sigma\left(C_{n}\right)=\left\{ 2\cdot cos\left(\frac{2\cdot\pi\cdot k}{n}\right)\ \mid\ k\in\left\langle n\right\rangle \right\} $\tabularnewline
\bottomrule
\end{tabular}
\par\end{center}

\begin{defn}
\label{def:co-spectrality-DS}$ $
\begin{enumerate}
\item Two graphs $G_{1}$ and $G_{2}$ are \emph{cospectral} if they have
the same spectrum.
\item A graph $G$ is \emph{determined by its spectrum} (DS) if every graph
that is cospectral with $G$ is isomorphic to $G$. 
\end{enumerate}
\end{defn}

There are many papers that give examples of DS graphs and of pairs
of non-isomorphic cospectral graphs, for example \cite{van2003graphs}
, \cite{van2009developments} , \cite{godsil1982constructing} .

\begin{thm}
The graphs $K_{n},C_{n},P_{n},\overline{K_{n}}$, and all the graphs
with less than $5$ vertices are DS \cite{van2003graphs}. 
\end{thm}

The following two examples are of pairs of non-isomorphic cospectral
graphs. 

\begin{example}
\cite{cvetkovic1980spectra} The graphs $S_{4}$ and $C_{4}\uplus K_{1}$
Figure \ref{fig:Cospectral-graphs-with-5-vertices} are cospectral
but \textbf{not} DS.
\begin{figure}[H]
\begin{centering}
\includegraphics[scale=0.3]{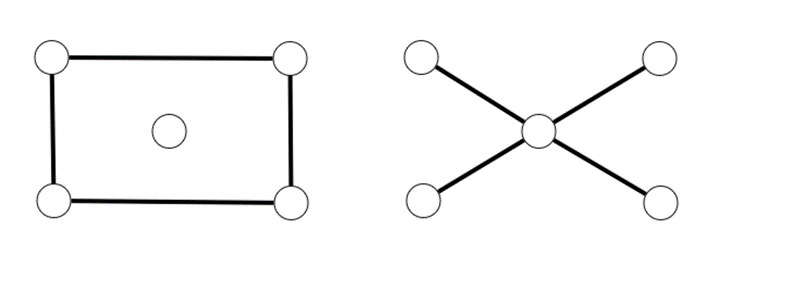}
\par\end{centering}
\caption{\label{fig:Cospectral-graphs-with-5-vertices}Non-isomorphic cospectral
graphs with $5$ vertices}
\end{figure}
\end{example}

The graphs are not isomorphic since one is connected and one is not.
In the following example both graphs are connected. 
\begin{example}
\label{exa:two-regular-cospectral-graphs} \cite{wananiyakul2022distinguish}
The two graphs $\Gamma_{1},\Gamma_{2}$ in Figure \ref{fig:r1_and_r2}
are cospectral and non-isomorphic. 
\begin{figure}[H]
\centering{}\includegraphics[scale=0.4]{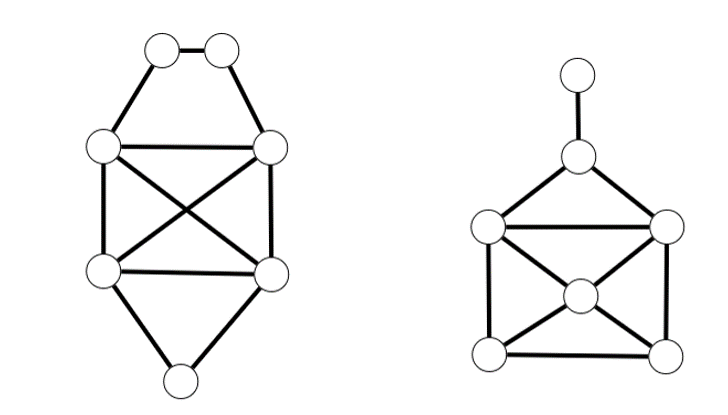}\caption{\label{fig:r1_and_r2} Non-isomorphic cospectral graphs with $7$
vertices}
\end{figure}
 
\end{example}

The following important result follows from Theorem \ref{thm:spectral-properties-of-a-graph}
and Corollary \ref{cor:num_of_adges_and_triangels}. 
\begin{thm}
\cite{van2003graphs} \label{thm:Deducted-properties-from-spectrum}
Let $G$ and $H$ be two cospectral graphs. Then 
\begin{enumerate}
\item $G$ and $H$ have the same number of edges. 
\item $G$ and $H$ have the same number of triangles. 
\item If $G$ is bipartite then $H$ is also bipartite.
\end{enumerate}
\end{thm}

\section{Completely positive matrices and graphs }
\begin{defn}
A matrix $A$ is \emph{completly positive} if there exists a nonnegative
matrix $B$ s.t. $A=BB^{T}$.
\end{defn}

The readers are referred to \cite{berman2003completely} and \cite{ShakedMonderer2021}
for the properties and the many applications of CP matrices. A necessary
condition for a matrix to be completely positive is that it is \emph{doubly nonnegative}
i.e. nonnegative and positive semidefinite. This condition is not
sufficient. 
\begin{example}
\label{exa:non-CP-matrix}The matrix
\[
\begin{pmatrix}1 & 1 & 0 & 0 & 1\\
1 & 2 & 1 & 0 & 0\\
0 & 1 & 2 & 1 & 0\\
0 & 0 & 1 & 2 & 1\\
1 & 0 & 0 & 1 & 3
\end{pmatrix}
\]
is doubly nonnegative but not CP. 
\end{example}

When is the necessary condition sufficient ? 

\begin{defn}
$ $ 
\begin{enumerate}
\item The graph of a matrix $A\in S^{n}$, denoted $G\left(A\right)$, is
a graph $G=\left(V,E\right)$ where 
\[
V=\left\langle n\right\rangle \ ,\ E=\left\{ \left(i,j\right)\ \mid\ i\ne j\ ,\ a_{i,j}\ne0\right\} .
\]
\item If $G\left(A\right)=G$, we say that $A$ is a \emph{matrix realization}
of $G$ .
\end{enumerate}
\end{defn}

\begin{example}
The graph of the matrix in Example \ref{exa:non-CP-matrix} is a pentagon. 
\end{example}

\begin{defn}
A graph $G$ is \emph{completly positive} ( \emph{CP graph} ) if
every doubly nonnegative matrix realization of $G$ is completely
positive. 
\end{defn}

Examples of CP graphs are : 
\begin{itemize}
\item Graphs with less than 5 vertices \cite{maxfield_minc_1962}.
\item Bipartite graphs \cite{berman1988bipartite}.
\item $T_{n}$ \cite{kogan1993characterization}.
\end{itemize}
\begin{defn}
A \emph{long odd cycle} is a cycle of odd length greater than 3. 
\end{defn}

Graphs that contain a long odd cycle are not completely positive \cite{berman1987combinatorial}.
\\

Examples of graphs that contain a long odd cycle are $B_{n}$ \cite{Berman2012}
, a cycle $C_{n}$ with a chord between vertices $1$ and $3$ (Figure
\ref{fig:H5-house-graph}), and the graphs $\Gamma_{1},\Gamma_{2}$
in Figure \ref{fig:r1_and_r2} . 

\begin{figure}[H]
\centering{}\includegraphics[scale=0.3]{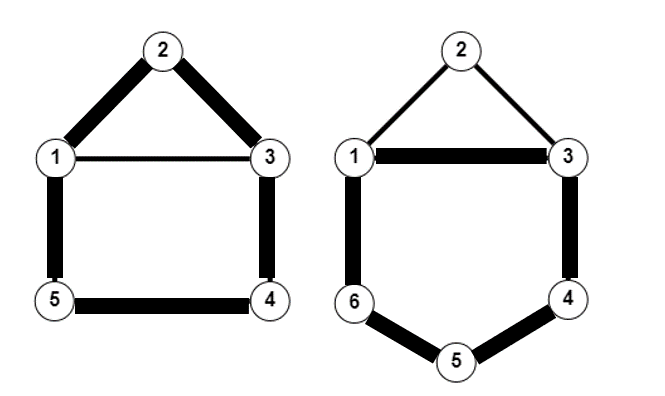}\caption{\label{fig:H5-house-graph}The graphs $B_{5}$ and $B_{6}$.}
\end{figure}

A characterization of completely positive graphs is : 
\begin{thm}
\label{thm:CP-of-graph} \cite{kogan1993characterization} The following
properties of a graph $G$ are equivalent :
\end{thm}

\begin{enumerate}
\item $G$ is CP. 
\item $G$ does not contain a long odd cycle
\item The line graph of $G$ is perfect. 
\end{enumerate}
The equivalence between 2 and 3 was proved in \cite{trotter1977line}
. 

\section{The graphs $T_{n}$ }

Recall that the graph $T_{n}$ consists of $n-2$ triangles with a
common base. 
\begin{thm}
\label{thm:The-spectrum-of-Tn}The spectrum of $T_{n}$ is 
\[
\sigma\left(T_{n}\right)=\left\{ \left(\frac{1-\sqrt{8n-15}}{2}\right),\left(-1\right),\left(0\right)^{n-3},\left(\frac{1+\sqrt{8n-15}}{2}\right)\right\} .
\]
\end{thm}

\begin{proof}
The adjacency matrix of $T_{n}$ is 
\[
A\left(T_{n}\right)=\left(\begin{array}{cccccc}
0 & 1 & 1 & 1 & \cdots & 1\\
1 & 0 & 1 & 1 & \cdots & 1\\
1 & 1 & 0 & 0 & \cdots & 0\\
\vdots & \vdots & 0 & 0 & \cdots & \vdots\\
1 & 1 & 0 & 0 & \cdots & 0\\
1 & 1 & 0 & 0 & \cdots & 0
\end{array}\right).
\]
Let $M=\lambda I-A\left(T_{n}\right)$ be the characteristic matrix
of $T_{n}$. The characteristic polynomial of $T_{n}$ is 
\[
det\left(M\right)=det\left(\begin{array}{cccccc}
\lambda & -1 & -1 & -1 & \cdots & -1\\
-1 & \lambda & -1 & -1 & \cdots & -1\\
-1 & -1 & \lambda & 0 & \cdots & 0\\
\vdots & \vdots & 0 & \lambda & 0 & 0\\
-1 & -1 & 0 & 0 & \ddots & 0\\
-1 & -1 & 0 & 0 & \cdots & \lambda
\end{array}\right).
\]

Denote $A=\begin{pmatrix}\lambda & -1\\
-1 & \lambda
\end{pmatrix}\ \ ,B=-J_{2,n}\ \ ,D=\lambda\cdot I_{n-2}$. Then
\[
M=\begin{pmatrix}A & B\\
B^{T} & D
\end{pmatrix}.
\]
 Using Theorem \ref{thm:Schur complement} , we get 
\[
det\left(M\right)=det\left(D\right)\cdot det\left(\nicefrac{M}{D}\right)=det\left(D\right)\cdot det\left(A-B^{T}D^{-1}B\right)
\]

\[
=det\left(\lambda\cdot I_{n-2}\right)\cdot det\left(A-J_{2,n}\cdot\frac{1}{\lambda}I_{n-2}\cdot J_{n,2}\right)
\]
\[
=\left(\lambda\right)^{n-2}\cdot det\left(\begin{pmatrix}\lambda & -1\\
-1 & \lambda
\end{pmatrix}-\frac{n-2}{\lambda}\cdot J_{2,2}\right)=\left(\lambda\right)^{n-2}\cdot det\begin{pmatrix}\lambda-\frac{n-2}{\lambda} & \frac{n-2}{\lambda}\\
\frac{n-2}{\lambda} & \lambda-\frac{n-2}{\lambda}
\end{pmatrix}=
\]
\[
=\left(\lambda\right)^{n-2}\cdot\left[\left(\lambda-\frac{n-2}{\lambda}\right)^{2}-\left(\frac{n-2}{\lambda}\right)^{2}\right]=\ \lambda\cdot\left(\lambda+1\right)\cdot\left(\lambda^{2}-\lambda+2\cdot\left(2-n\right)\right).
\]
Hence, the spectrum is: 
\[
\sigma\left(T_{n}\right)=\left\{ \left(\frac{1-\sqrt{8n-15}}{2}\right),\left(-1\right),\left(0\right)^{n-3},\left(\frac{1+\sqrt{8n-15}}{2}\right)\right\} .
\]
\end{proof}
\begin{rem}
Observe that for $n=3$ we get the spectrum of a triangle $\sigma\left(T_{3}\right)=\sigma\left(K_{3}\right)=\sigma\left(C_{3}\right)=\left\{ \left(-1\right)^{2},\left(2\right)\right\} $.
\end{rem}

In the next section we study the graphs $T_{n,k}$, and show that
for $k>1$ they are DS (Theorem \ref{thm:MAIN-Tnk-DS}). Since $T_{n}=T_{n,2}$,
we get the following theorem as special case of Theorem \ref{thm:MAIN-Tnk-DS}. 
\begin{thm}
\label{thm:MAIN-Tn-DS}The graph $T_{n}$ is determined by its spectrum. 
\end{thm}

In the previous section we showed that $T_{n}$ is CP. For $k>2$,
$T_{n,k}$ contains a long odd cycle, so they are not CP. Hence $T_{n}$
are the only graphs of pyramids that are both DS and CP. 

\section{The graphs $T_{n,k}$}

Let $k<n$ be natural numbers. Define $T_{n,k}=K_{k}\lor\overline{K_{n-k}}$
(not to be confused with Turan graph $T\left(n,k\right)$). For $k=1$,
$T_{n,1}=S_{n-1}$, and for $k=2$, $T_{n,2}=T_{n}$.

Since the graph $T_{n,k}$, $k>1$ consists of $n-k$ pyramids ($K_{k+1}$)
that intersect in $K_{k}$ we refer to it as a \emph{graph of pyramids}.
\begin{thm}
\label{claim:spectrum-of-Tnk} The spectrum of $T_{n,k}$ is 
\[
\sigma\left(T_{n,k}\right)=\left\{ \left(\lambda_{2}\right),\left(-1\right)^{k-1},\left(0\right)^{n-k-1},\left(\lambda_{1}\right)\right\} 
\]
where $\lambda_{1},\lambda_{2}=\frac{\left(k-1\right)\pm\sqrt{\left(k-1\right)^{2}+4k\left(n-k\right)}}{2}$. 
\end{thm}

\begin{rem}
We precede the proof with a sanity check. For $k=1$, we get the spectrum
of $S_{n-1}$ (Theorem \ref{thm:spectrum-of-well}). For $k=2$, we
get the spectrum of $T_{n}$ (Theorem \ref{thm:The-spectrum-of-Tn}).
\end{rem}

\begin{proof}
(of Theorem \ref{claim:spectrum-of-Tnk}) Let $k\le n$ be natural
numbers and denote $r:=n-k$. The adjacency matrix of $T_{n,k}$ is
\[
A\left(T_{n,k}\right)=\begin{pmatrix}J_{k}-I_{k} & J_{k,r}\\
J_{r,k} & 0_{r}
\end{pmatrix}.
\]
Let $M=\lambda I_{n}-A\left(T_{n,k}\right)$ be the characteristic
matrix of $T_{n,k}$. The characteristic polynomial of $T_{n,k}$
is 
\[
det\left(M\right)=det\begin{pmatrix}\left(\lambda+1\right)I_{k}-J_{k} & -J_{k,r}\\
-J_{r,k} & \lambda I_{r}
\end{pmatrix}.
\]
 Denote $B=\left(\lambda+1\right)I_{k}-J_{k}$. Then 
\[
M=\begin{pmatrix}B & -J_{k,r}\\
-J_{r,k} & \lambda I_{r}
\end{pmatrix}.
\]
 The Schur complement $\nicefrac{M}{\lambda I_{r}}$ is 
\[
\nicefrac{M}{\lambda I_{r}}=B-J_{k,r}\cdot\left(\lambda I\right)^{-1}\cdot J_{r,k}=B-\frac{1}{\lambda}J_{k,r}\cdot J_{r,k}=B-\frac{r}{\lambda}J_{k}=
\]
\[
=\left(\lambda+1\right)I_{k}-J_{k}-\frac{r}{\lambda}J_{k}=\left(\lambda+1\right)I_{k}-\frac{r+\lambda}{\lambda}J_{k}.
\]
 By the theorem of Schur (Theorem \ref{thm:Schur complement} )
\[
det\left(M\right)=det\left(\lambda I_{r}\right)\cdot det\left(\nicefrac{M}{\lambda I_{r}}\right)=\lambda^{r}\cdot det\left(\nicefrac{M}{\lambda I_{r}}\right).
\]
 The determinant of the Schur complement is : 
\[
det\left(\nicefrac{M}{\lambda I_{r}}\right)=det\left(\left(\lambda+1\right)I_{k}-\frac{r+\lambda}{\lambda}J_{k}\right)=det\left(\frac{r+\lambda}{\lambda}\cdot\left(\frac{\lambda\cdot\left(\lambda+1\right)}{\lambda+r}I_{k}-J_{k}\right)\right)
\]
\[
=\left(\frac{r+\lambda}{\lambda}\right)^{k}\cdot det\left(\frac{\lambda\cdot\left(\lambda+1\right)}{\lambda+r}I_{k}-J_{k}\right).
\]
 Since $\sigma\left(J_{k}\right)=\left\{ \left(0\right)^{k-1},\left(k\right)\right\} $,
the spectrum of $\frac{\lambda\cdot\left(\lambda+1\right)}{\lambda+r}I_{k}-J_{k}$
is 
\[
\sigma\left(\frac{\lambda\cdot\left(\lambda+1\right)}{\lambda+r}I_{k}-J_{k}\right)=\left\{ \left(\frac{\lambda\cdot\left(\lambda+1\right)}{\lambda+r}-k\right),\left(\frac{\lambda\cdot\left(\lambda+1\right)}{\lambda+r}\right)^{k-1}\right\} ,
\]
so 
\[
det\left(\frac{\lambda\cdot\left(\lambda+1\right)}{\lambda+r}I_{k}-J_{k}\right)=\left(\frac{\lambda\cdot\left(\lambda+1\right)}{\lambda+r}-k\right)\cdot\left(\frac{\lambda\cdot\left(\lambda+1\right)}{\lambda+r}\right)^{k-1}
\]
\[
det\left(\nicefrac{M}{\lambda I_{r}}\right)=\left(\frac{r+\lambda}{\lambda}\right)^{k}\cdot\left(\frac{\lambda\cdot\left(\lambda+1\right)}{\lambda+r}-k\right)\cdot\left(\frac{\lambda\cdot\left(\lambda+1\right)}{\lambda+r}\right)^{k-1}=
\]
\[
=\frac{\left(\lambda+1\right)^{k-1}\cdot\left(\lambda\cdot\left(\lambda+1\right)-k\left(\lambda+r\right)\right)}{\lambda}.
\]
 Hence, 
\[
det\left(M\right)=\lambda^{r}\cdot det\left(\nicefrac{M}{\lambda I_{r}}\right)=\lambda^{r}\cdot\left(\frac{\left(\lambda+1\right)^{k-1}\cdot\left(\lambda\cdot\left(\lambda+1\right)-k\left(\lambda+r\right)\right)}{\lambda}\right)
\]
\[
=\lambda^{r-1}\cdot\left(\lambda+1\right)^{k-1}\cdot\left(\lambda\left(\lambda+1\right)-k\left(\lambda+r\right)\right)
\]
\[
=\lambda^{r-1}\cdot\left(\lambda+1\right)^{k-1}\cdot\left(\lambda^{2}+\left(1-k\right)\lambda-rk\right).
\]
 Thus, the eigenvalues of $T_{n,k}$ are $0$ with multiplicity $n-k-1$,
$-1$ with multiplicity $k-1$ and the two roots of $\lambda^{2}+\left(1-k\right)\lambda-\left(n-k\right)k$.
\end{proof}
Is $T_{n,k}$ determined by its spectrum ? 

We consider two cases - the star graphs and the graphs of pyramids.

\begin{thm}
The star graph $S_{n}=T_{n+1,1}$ is DS iff $n$ is prime. 
\end{thm}

\begin{proof}
The spectrum of $S_{n}$ is 
\[
\sigma\left(S_{n}\right)=\left\{ \left(-\sqrt{n}\right),\left(0\right)^{n-1},\left(\sqrt{n}\right)\right\} 
\]
\begin{itemize}
\item If $n$ is composite, there exist natural numbers $2\le p\le q$ s.t.
$n=p\cdot q$. Let $l=n+1-p-q$. Observe that $1\le l$ since 
\[
n=p\cdot q=\left(q-1\right)p+p\ge\left(2-1\right)p+p=p+p\ge q+p.
\]
Let $H=K_{p,q}\uplus\overline{K_{l}}$. Since $\sigma\left(K_{p,q}\right)=\left\{ \left(-\sqrt{pq}\right),\left(0\right)^{p+q-2},\left(\sqrt{pq}\right)\right\} $
and $\sigma\left(\overline{K_{l}}\right)=\left\{ \left(0\right)^{l}\right\} $,
$H$ and $S_{n}$ are cospectral. They are not isomorphic since $S_{n}$
is connected and $H$ is not connected.
\item If $n$ is prime, we have to show that $S_{n}$ is DS. Let $G=\left(V,E\right)$
be cospectral with $S_{n}$. Since all the graphs of less than $5$
vertices are DS, we can assume that $n\ge5$. Since $S_{n}$ is bipartite,
$G$ is bipartite as well, so the subgraph induced by a maximal clique
in $G$ is $K_{2}$. $G$ is not a disjoint union of non-empty graphs,
since it has only $1$ positive eigenvalue. By Theorem \ref{thm:induced-subgraph-interlacing},
$P_{4}$ is not an induced subgraph of $G$, since $rank\left(A\left(P_{4}\right)\right)>2$.
Hence, the maximal distance between two vertices in the same connected
component of $G$ is $2$. First we show that $G$ is a tree by showing
it has no cycles, and then that the trees $G$ and $S_{n}$ are isomorphic.
Suppose to the contrary that there exists a cycle in $G$. Let $H$
be the only non-empty connected component of $G$. Since $G$ is bipartite,
$H$ is bipartite as well. Let $L$ and $R$ be the two parts of $H$,
$H=\left(L\uplus R,E\right)$. Since $H$ has a cycle $\left|L\right|>1$,
$\left|R\right|>1$. 
\begin{itemize}
\item If $H$ is a complete bipartite graph then 
\[
\sigma\left(H\right)=\left\{ \left(-\sqrt{\left|L\right|\cdot\left|R\right|}\right),\left(0\right)^{\left|L\right|+\left|R\right|-2},\left(\sqrt{\left|L\right|\cdot\left|R\right|}\right)\right\} .
\]
Thus
\[
\sigma\left(G\right)=\left\{ \left(-\sqrt{\left|L\right|\cdot\left|R\right|}\right),\left(0\right)^{\left|V\right|-2},\left(\sqrt{\left|L\right|\cdot\left|R\right|}\right)\right\} =\sigma\left(S_{n}\right)
\]
\[
\Rightarrow\left|L\right|\cdot\left|R\right|=n,
\]
 contradicting the primality of $n$. 
\item If $H$ is not a complete bipartite graph, then there exist $v\in L,u\in R$
s.t. $\left(v,u\right)\notin E$. Hence, $d\left(u,v\right)\ne1$.
Since $G$ is bipartite $d\left(u,v\right)\ne2$. Since $u,v$ are
in the same connected component, $d\left(v,u\right)>2$. a contradiction. 
\end{itemize}
We showed that $G$ has no cycle, so since $G$ has $n$ edges it
is a tree by Theorem \ref{thm:Tree-properties}. Let $v_{1},v_{2},v_{3}$
be vertices of $G$ where $v_{1},v_{2}$ are neighbors. Since $G$
is triangle-free we can assume without loss of generality that $\left(v_{2},v_{3}\right)\notin E$.
Assume to the contrary that there exists a vertex $u$ s.t. $\left(u,v_{1}\right)\notin E$.
Since $G$ is a tree there is only one path from $v_{1}$ to $u$.
Let $w\in V$ be a vertex s.t. $\left(w,u\right),\left(v_{1},u\right)\in E$.
Assume without loss of generality that $w\ne v_{2}$. The distance
between $w$ and $v_{2}$ is strictly greater than $2$ - a contradiction
since $P_{4}$ is not an induced subgraph of $G$. Hence, $v_{1}$
is adjacent to all the other vertices, so $G$ is isomorphic to $S_{n}$. 
\end{itemize}
\end{proof}
\begin{thm}
\label{thm:MAIN-Tnk-DS}The graphs of pyramids are DS.
\end{thm}

\begin{proof}
Let $1<k\le n$ and let $G=\left(V,E\right)$ be a graph that is cospectral
with $T_{n,k}$. We have to show that $G$ and $T_{n,k}$ are isomorphic.
We start by showing that 
\begin{itemize}
\item $K_{k+2}$ and $P_{4}$ are not induced subgraphs of $G$ (Lemma \ref{lem:k+2 clique}). 
\item Every two non-empty induced subgraphs of $G$ are connected by an
edge (Lemma \ref{lem:edge-disjoind-induced-cliques}).
\item If $U$ is a maximal clique in $G$, then every edge in $G$ has an
end in $U$ (Lemma \ref{lem:spanning-clique}). 
\end{itemize}
\begin{lem}
\label{lem:properties_of_G_Tnk}$ $
\begin{enumerate}
\item $G$ has one positive eigenvalue. 
\item $G$ has $k$ negative eigenvalues.
\item $G$ has one eigenvalue strictly smaller than $-1$.
\end{enumerate}
\end{lem}

\begin{proof}
Since $G$ is cospectral with $T_{n,k}$, its spectrum is 
\[
\sigma\left(T_{n,k}\right)=\left\{ \left(\lambda_{2}\right),\left(-1\right)^{k-1},\left(0\right)^{n-k-1},\left(\lambda_{1}\right)\right\} 
\]
where $\lambda_{1},\lambda_{2}=\frac{\left(k-1\right)\pm\sqrt{\left(k-1\right)^{2}+4k\left(n-k\right)}}{2}$. 
\end{proof}
\begin{lem}
\label{lem:k+2 clique}$ $
\begin{enumerate}
\item $G$ does not contain a clique of size $k+2$.
\item $G$ does not have $P_{4}$ as an induced subgraph.
\end{enumerate}
\end{lem}

\begin{proof}
By Theorem \ref{thm:spectrum-of-well} , the spectrum of $K_{k+2}$
is $\sigma\left(K_{k+2}\right)=\left\{ \left(-1\right)^{k+1},\left(k+1\right)\right\} $,
and the spectrum of $P_{4}$ is $\sigma\left(P_{4}\right)=\left\{ 2\cdot cos\left(\frac{\pi\cdot k}{5}\right)\ \mid\ k\in\left\langle 4\right\rangle \right\} $.
$K_{k+2}$ has $k+1$ negative eigenvalues and $P_{4}$ has $2$ positive
eigenvalues. By Lemma \ref{lem:properties_of_G_Tnk}, $G$ has $k$
negative eigenvalues and $1$ positive eigenvalue. Thus, by Theorem
\ref{thm:induced-subgraph-interlacing}, $P_{4}$ and $K_{k+2}$ are
not induced subgraphs of $G$.
\end{proof}
\begin{lem}
\label{lem:edge-disjoind-induced-cliques}Every two disjoint non-empty
induced subgraphs in $G$ are connected by an edge. 
\end{lem}

\begin{proof}
Assume there exist two disjoint non-empty induced subgraphs $H_{1}=\left(V_{1},E_{1}\right),H_{2}=\left(V_{2},E_{2}\right)$.
Let $H=H_{1}\uplus H_{2}$ be the subgraph induced by $V_{1}\cup V_{2}$.
\[
A\left(H\right)=\begin{pmatrix}A\left(H_{1}\right) & 0\\
0 & A\left(H_{2}\right)
\end{pmatrix}.
\]
By Theorem \ref{claim:Spectrum-of-disjoint}, $\sigma\left(H\right)=\sigma\left(H_{1}\right)+\sigma\left(H_{2}\right)$.
Both $H_{1}$ and $H_{2}$ have a positive eigenvalue, since they
are non-empty. Thus, $H$ has at least two positive eigenvalues, contradicting
Theorem \ref{thm:induced-subgraph-interlacing} . 
\end{proof}
\begin{lem}
\label{lem:spanning-clique}Let $H_{1}=\left(V_{1},E_{1}\right)$
be a maximal clique in $G$, and let $H_{2}=\left(V_{2},E_{2}\right)$
be the subgraph induced by the vertices $V_{2}:=V\setminus V_{1}$.
Then $H_{2}$ is an empty graph. 
\end{lem}

\begin{proof}
Without loss of generality we assume that $V_{1}=\left\langle m\right\rangle $.
By Lemma \ref{lem:k+2 clique} $m\le k+1$. We have to show that $H_{2}$
is an empty graph. Assume to the contrary that there is an edge $e=\left(v_{1},v_{2}\right)\in E_{2}$,
for some $v_{1},v_{2}\in V_{2}$. Since $H_{1}$ is induced by a maximal
clique, there exists a vertex $u_{1}\in V_{1}$ s.t. $\left(v_{1},u_{1}\right)\notin E$
(otherwise, the subgraph obtained by $V\uplus\left\{ v_{1}\right\} $
is a complete graph of order greater than $m$, contradicting the
maximality of $H_{1}$). Similarly, there exists a vertex $u_{2}\in V_{1}$
s.t. $\left(v_{2},u_{2}\right)\notin E$. We can also assume that
$u_{1}\ne u_{2}$, since if both $v_{1},v_{2}$ are connected to all
the vertices except $u_{1}$, then $\left(V_{1}\setminus\left\{ u\right\} \right)\uplus\left\{ v_{1},v_{2}\right\} $
is a clique, contradicting the maximality of $V_{1}$. 

We now show that the following cases are impossible : 
\begin{enumerate}
\item $\left(v_{1},u_{2}\right)\notin E$ and $\left(v_{2},u_{1}\right)\notin E$
\item $\left(v_{1},u_{2}\right)\in E$ but $\left(v_{2},u_{1}\right)\notin E$ 
\item $\left(v_{1},u_{2}\right)\notin E$ but $\left(v_{2},u_{1}\right)\in E$ 
\item $\left(v_{1},u_{2}\right)\in E$ and $\left(v_{2},u_{1}\right)\in E$
\end{enumerate}
Case $1$ is impossible since in this case the subgraph obtained by
the vertices $\left\{ v_{1},v_{2},u_{1},u_{2}\right\} $ is $K_{2}\uplus K_{2}$.
Thus, $G$ has induced cliques that are not connected by an edge,
contradicting Lemma \ref{lem:edge-disjoind-induced-cliques}. Case
$2$ is impossible since the subgraph induced by the vertices $\left\{ v_{1},v_{2},u_{1},u_{2}\right\} $
is $P_{4}$, contradicting Lemma \ref{lem:k+2 clique}. Similarly,
Case $3$ is impossible. The proof that Case $4$ is impossible needs
more work:

Since $G$ is cospectral with $T_{n,k}$ they have the same number
of edges. Thus, there exists another edge in $E_{2}$. By Lemma \ref{lem:edge-disjoind-induced-cliques},
$G$ does not have disjoint cliques. Thus, there is an edge $\left(v_{1},v_{3}\right)$
for some $v_{3}\in V_{2}$. By Lemma \ref{lem:k+2 clique} there exists
$u_{3}\in V_{1}$ s.t. $\left(u_{3},v_{3}\right)\notin E$. Up to
isomorphism, there are two possibilities for the subgraph induced
by vertices $\left\{ v_{1},v_{2},v_{3},u_{1},u_{2},u_{3}\right\} $
- $\left(v_{2},v_{3}\right)\in E_{2}$ or $\left(v_{2},v_{3}\right)\notin E_{2}$. 

If $\left(v_{2},v_{3}\right)\in E_{2}$, then $\left\{ v_{1},v_{2},v_{3}\right\} $
is a triangle in $H_{2}$. Since $H_{1}$ is a maximal clique, there
exists a vertex $u_{3}\in V_{1}$ s.t. $\left(v_{3},u_{3}\right)\notin E$
(otherwise $\left\{ v_{1},v_{3}\right\} $ or $\left\{ v_{2},v_{3}\right\} $
satisfies one of the previous conditions). In this case, the subgraph
induced by the vertices $\left\{ v_{1},v_{2},v_{3},u_{1},u_{2},u_{3}\right\} $
is $H_{3}$ in Figure \ref{fig:h3}. Its adjacency matrix is 
\[
A\left(H_{3}\right)=\begin{pmatrix}0 & 1 & 1 & 0 & 1 & 1\\
1 & 0 & 1 & 1 & 0 & 1\\
1 & 1 & 0 & 1 & 1 & 0\\
0 & 1 & 1 & 0 & 1 & 1\\
1 & 0 & 1 & 1 & 0 & 1\\
1 & 1 & 0 & 1 & 1 & 0
\end{pmatrix}
\]
 and its spectrum is 
\[
\sigma\left(H_{3}\right)=\left\{ \left(0\right)^{3},\left(-2\right)^{2},\left(4\right)\right\} .
\]
\begin{figure}[H]
\centering{}\includegraphics[scale=0.4]{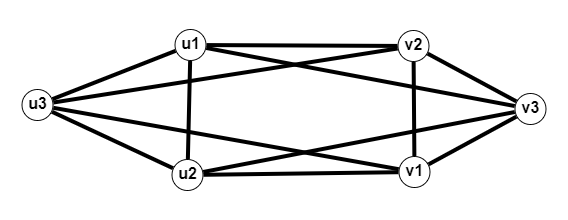}\caption{\label{fig:h3}$H_{3}$}
\end{figure}
Hence, $H_{3}$ has $2$ eigenvalues strictly smaller than $-1$,
contradicting Theorem \ref{thm:induced-subgraph-interlacing}. 

If $\left(v_{2},v_{3}\right)\notin E_{2}$, there exists a vertex
$u_{3}\in V_{1}$ s.t. $\left(v_{3},u_{3}\right)\notin E$. Assume
that there exists $u_{3}\notin\left\{ u_{1},u_{2}\right\} $ such
that $\left(v_{1},u_{3}\right)\in E$ and $\left(v_{2},u_{3}\right)\in E$
(otherwise one of the previous conditions occur, with respect to $v_{3}$).
In this case, the subgraph induced by the vertices $\left\{ v_{1},v_{2},v_{3},u_{1},u_{2},u_{3}\right\} $
is $H_{4}$ in Figure \ref{fig:H4}. Its adjacency matrix is 
\[
A\left(H_{4}\right)=\begin{pmatrix}0 & 1 & 1 & 0 & 1 & 1\\
1 & 0 & 0 & 1 & 0 & 1\\
1 & 0 & 0 & 1 & 1 & 0\\
0 & 1 & 1 & 0 & 1 & 1\\
1 & 0 & 1 & 1 & 0 & 1\\
1 & 1 & 0 & 1 & 1 & 0
\end{pmatrix}
\]
 and its spectrum is 
\[
\sigma\left(H_{4}\right)=\left\{ \left(\frac{-\sqrt{5}-1}{2}\right),\left(-2.231\right),\left(-0.483\right),\left(0\right),\left(\frac{-\sqrt{5}-1}{2}\right),\left(3.714\right)\right\} .
\]
\begin{figure}[H]
\begin{centering}
\includegraphics[scale=0.4]{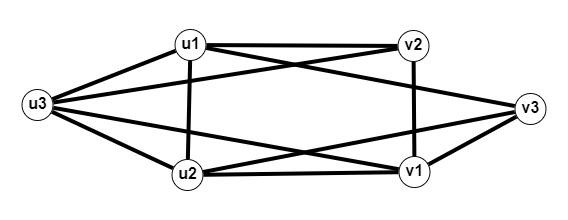}
\par\end{centering}
\caption{\label{fig:H4}$H_{4}$}
\end{figure}
 Hence, $H_{4}$ also has $2$ eigenvalues strictly smaller than $-1$,
contradicting Theorem \ref{thm:induced-subgraph-interlacing} . 
\end{proof}
Now we can complete the proof of the theorem. The adjacency matrix
of $G$ is
\[
A\left(G\right)=\begin{pmatrix}A\left(H_{1}\right) & B\\
B^{T} & 0
\end{pmatrix}=\begin{pmatrix}J_{m}-I_{m} & B\\
B^{T} & 0_{n-m}
\end{pmatrix}
\]
 for some block matrix $B$. Recall that 
\[
A\left(T_{n,k}\right)=\begin{pmatrix}J_{k+1}-I_{k+1} & J_{k+1,n-k-1}\\
J_{n-k-1,k+1} & 0_{n-k-1}
\end{pmatrix}.
\]
 We want to show that $m=k+1$ and $B=J_{n\times\left(n-m\right)}$.
Assume to the contrary $m\ne k+1$. By Lemma \ref{lem:k+2 clique},
$m<k+1$. Then clearly the number of edges of $G$ is smaller then
the number of edges of $T_{n,k}$, contradicting Theorem \ref{thm:Deducted-properties-from-spectrum}.
Hence, $m=k+1$ and for some $B\in\mathbb{R}^{\left(k+1\right)\times\left(n-k-1\right)}$
\[
A\left(G\right)=\begin{pmatrix}J_{k+1}-I_{k+1} & B\\
B^{T} & 0_{n-k-1}
\end{pmatrix}.
\]
Again by considering the number of edges, all the entries of $B$
are equal to $1$. Thus 
\[
A\left(G\right)=\begin{pmatrix}J_{k+1}-I_{k+1} & J_{k+1,n-k-1}\\
J_{n-k-1,k+1} & 0_{n-k-1}
\end{pmatrix}=A\left(T_{n,k}\right).
\]
\end{proof}

\section{Summary}

We showed that the graphs of pyramids $T_{n,k}$ ($2\le k<n$) are
DS (Theorem \ref{thm:MAIN-Tnk-DS}), and that for $k=1$, the graph
$T_{n,1}=S_{n-1}$ is DS if and only if $n-1$ is prime.

In the next table we list the graphs $T_{n,k}$ according to their
being DS/CP :

\begin{table}[H]
\centering{}%
\begin{tabular}{ccc}
\toprule 
 & CP & not CP\tabularnewline
\midrule 
DS & $k=2$ or $k=1,\ n$ is prime & $3\le k$\tabularnewline
\midrule 
not DS & $k=1,\ n$ is composite & \tabularnewline
\bottomrule
\end{tabular}\caption{Classification of the graphs $T_{n,k}$ according to their being DS/CP.}
\end{table}

We have an empty field in the table, since the only graphs $T_{n,k}$
that are not DS are stars, and stars are CP since they are bipartite.
This suggests the following questions on general graphs (not only
$T_{n,k}$). 
\begin{question}
What is the smallest order $\nu$ of a graph that is not CP neither
DS ? 
\end{question}

Solution. The graphs $\Gamma_{1},\Gamma_{2}$ in Figure \ref{fig:r1_and_r2}
are not DS, and by Theorem \ref{thm:CP-of-graph} are not CP, so $\nu\le7$.
On the other hand, $\nu\ge6$ since all the graphs with 4 vertices
are CP, and there is only one pair of non-isomorphic cospectral graphs
with 5 vertices (Figure \ref{fig:Cospectral-graphs-with-5-vertices}),
and both are CP. Computer search shows that $\nu>6$ \cite{cvetkovic1984table}
, so $\nu=7$. 

\bibliographystyle{plain}
\bibliography{refs}

\end{document}